\documentclass[10pt]{article}
\usepackage{amsthm}
\usepackage{amsfonts}
\usepackage{amssymb}
\usepackage{amsmath}
\usepackage{nccmath}
\usepackage{mathtools}
\usepackage{mathrsfs}
\usepackage{setspace}
\usepackage{graphicx}
\usepackage[export]{adjustbox}
\usepackage{hyperref}
\usepackage[numbers,sort&compress]{natbib}
\usepackage{enumerate}
\usepackage{authblk}
\usepackage{placeins}

\newtheorem{theorem}{Theorem}[section]

\theoremstyle{definition}

\newtheorem{example}[theorem]{Example}

\theoremstyle{remark}
\newtheorem{remark}[theorem]{Remark}

\numberwithin{equation}{section}

\usepackage{amsmath,amssymb,graphicx} 
\usepackage[margin=1in]{geometry} 
\usepackage{enumerate}
\usepackage{authblk}
\usepackage{placeins}
\usepackage{array}
\usepackage{academicons}
\usepackage{xcolor}
\usepackage{svg}
\usepackage[utf8]{inputenc}
\usepackage{booktabs}
\usepackage{tikz}
\usepackage{placeins}
\usepackage{balance}
\usepackage{longtable}

\providecommand{\keywords}[1]{\textbf{\textit{Keywords:}} #1}
\providecommand{\subjclass}[1]{\textbf{\textit{MSC2020:}} #1}
\begin{document}

\nocite{*} 

\title{ On the Generalization of Weinberger's Inequality with Alternating Signs}

\author{Hailu Bikila Yadeta \\ email: \href{mailto:haybik@gmail.com}{haybik@gmail.com} }
  \affil{Salale University, College of Natural Sciences, Department of Mathematics\\ Fiche, Oromia, Ethiopia}
\date{\today}
\maketitle
\noindent
\begin{abstract}
\noindent For  given set of $m$ positive numbers satisfying the conditions:
$$ a_1  \geq a_2 \geq , ... \geq a_m  \geq 0, $$
the inequality
$$   \sum_{s=1}^{m} (-1)^{s-1}a^r_s  \geq \left[ \sum_{s=1}^{m} (-1)^{s-1}a_s\right]^r, \quad r > 1,  $$
was proved by H. Weinberger. 
The generalization of  Weinberger's result takes the form
$$   \sum_{s=1}^{m} (-1)^{s-1}f(a_s)  \geq f\left( \sum_{s=1}^{m} (-1)^{s-1}a_s\right),  $$
where $f$ is a convex function satisfying the condition $f(0)\leq 0 $. The condition  $f(0)\geq 0 $ in the generalization proposed by Bellman was corrected by Olkin as $f(0) \leq 0 $. Bellman gave only  a graphical proof for differentiable convex functions.
In this paper, we give a mathematical proof for the generalized inequality  including the importance of the condition $f(0)\leq 0$.  We introduce a set $\mathcal{W}$ of functions so that functions in the intersection of $\mathcal{W}$  and the set of all convex functions are  the ones that are desirable in the generalization.  In addition, we give a proof of Szeg\"{o}'s inequality which applies to sums with odd number of terms.
\end{abstract}

\noindent\keywords{inequality, alternating terms, convex function, monotone increasing function, floor function }\\
\\

\section{Introduction}
 H. Weinberger  \cite{HW} proved an inequality with alternating signs.
\begin{theorem}[\textbf{H. Weinberger}]\label{eq:Weinberger's}
    For  given set of $m$ positive numbers satisfying the conditions
    \begin{equation}\label{eq:nonnegativedomaincondition}
    a_1  \geq a_2 \geq , ... \geq a_m \geq 0,
    \end{equation}
the inequality
\begin{equation}\label{eq:WinbergerInequality}
  \sum_{s=1}^{m} (-1)^{s-1}a^r_s  \geq \left[ \sum_{s=1}^{m} (-1)^{s-1}a_s\right]^r, \quad r > 1
\end{equation}
holds true.
\end{theorem}
\noindent  Weinberger mentioned that the problem was originally conjectured by L. E. Pyne and A. Weinstein. In addition to the proof of the inequality (\ref{eq:WinbergerInequality}), Weinberger provided some geometric significance to the inequality. A similar geometric application in the generalized theory of symmetrization of Payne and Weintein, who consider bodies of revolution in a fictitious space of non-integral dimensions, is cited in \cite{HW}. Some other authors treat inequalities  involving alternating sums like that of Weinberger's inequality. For example,
\begin{theorem}[\textbf{Szeg\"{o} } ]\cite{IO}
  Let $a_1 > a_2>...> a_{2m-1} > 0 $, and $f(x)$ is a convex function defined on $ [0,a_1]$. Then
\begin{equation}\label{eq:szegoinequality}
     \sum_{j=1}^{2m-1}(-1)^{j-1} f(a_j)\geq f\left[  \sum_{j=1}^{2m-1}(-1)^{j-1}a_{j} \right],
  \end{equation}
\end{theorem}

\begin{theorem}[\textbf{Bellman}]\cite{RB}
Let $f$ be a function satisfying the conditions:
\begin{equation}\label{eq:Bellmanscondition}
  f(0) \geq 0, \quad  f'(0) \geq 0, \quad  f'(x)   \text{ is monotone increasing },
\end{equation}
then
\begin{equation}\label{eq:generalizedinequality}
    \sum_{n=1}^{k}(-1)^{n-1} f(a_n)\geq f\left[  \sum_{n=1}^{k}(-1)^{n-1}a_{n} \right] .
\end{equation}

  provided that (\ref{eq:nonnegativedomaincondition}) is valid.
\end{theorem}
\noindent Ingram Olkin \cite{IO} cites the conditions (\ref{eq:Bellmanscondition}) used by Bellman \cite{RB} as:
\begin{equation}\label{eq:Ingramsconditions}
   f   \text{ is a convex function  defined on } [0,a_1],\quad   f(0) \leq 0.
 \end{equation}
It is true that a differentiable function is convex if and only if its derivative is  monotone increasing. See, for example, \cite{KG}. The condition  $f(0) \leq 0 $, which corrects Bellman's result,  was introduced by Olkin  \cite{IO}. That one used by Bellman was $f(0) \geq 0 $. The inequality by Szeg\"{o} considers  general convex functions, which includes the Weinberger's $f(x)= x^r , r > 1 $. While both Weinberger and Szeg\"{o} consider convex functions, Szeg\"{o}'s  does not treat inequalities involving summations with even number terms. For an even number of terms, the inequality  may work for some convex functions, for example, $ f(x)= x^2 + x^4 +x^6 $ but not for others, for example, $f(x)= e^x $. To summarize, Szeg\"{o}'s work, earlier than Weinbergers's result, considers general class of convex functions with inequalities involving an odd number of terms of summation, while Weinberger's result treats a lesser class of convex functions namely, $f(x)= x^r , r> 1 $ without restriction on the number of terms in the summations in the inequalities. Weinberger's inequality works for both those summations with even, and odd numbers of terms.  For Bellman's generalization, by the fact that differentiable functions with monotone increasing derivatives are convex, the generalization seems to be similar to that of  Szeg\"{o}'s result, except that two additional conditions  $f(0), f'(0) \geq 0 $ are included by Bellman. We shall give a counter example on how differentiable convex function with these two conditions at $x=0$ fails to meet the generalized inequality proposed by Bellman.  Olkin claims a generalized inequality that incorporates the works of Szeg\"{o}, Weinberger, and Bellman. See  \cite{IO}. Latter,  Bellman \cite{RB1} shows  that Olkin's result in \cite{IO}  is a special case of an inequality due to Steffensen \cite{ST}. While Bellman's generalization with the correct condition
$f(0) \leq 0 $ by Olkin is true, the mathematical analyses of the result were not displayed. Belman uses geometric consideration.
 \noindent In this paper, we introduce a set $\mathcal{W}$  that contains  the set of all functions of the form $ f(x)=x^r,\, r> 1$ that were used by Weinberger as a subset.  While not all convex functions are in the class $\mathcal{W}$, those convex functions in the class $\mathcal{W}$ play an important role in the generalization of Weiberger's  inequality. One of the important properties of the function  $ f \in \mathcal{W}$ is that $f(0)\leq  0 $, and  the only convex functions in the class  $ f \in \mathcal{W}$  are the ones with the property $f(0) \leq 0 $. The current work treats the following  issues:
\begin{itemize}
  \item  Szeg\"{o} inequality \cite{GS} was cited by  Olkin \cite{IO}. Here we present a proof of Szeg\"{o} inequality with a simple procedure that only takes into account the convexity of the function $f$.
  \item We introduce the class $ \mathcal{W}$ of functions such that the set of all functions that are used in the generalization of the inequality are those which belong to
      $$ \{f:  f \text{ is convex}\} \cap \mathcal{W}  . $$
      There are convex function  that are not in the class $ \mathcal{W}$, and functions in the class $ \mathcal{W}$ that are not convex. We show that
      $$ \{f:  f \text{ is convex }\} \cap \mathcal{W} = \{f:  f \text{ is convex},\quad f(0)\leq 0 \}   $$


  \item The condition $f(0) \leq 0 $ is very important for the generalizations of Weinberger inequality. This condition was wrongly written as $f(0) \geq 0 $ by Bellman and was corrected by Olkin as to be $f(0) \leq 0 $. However Olkin has given a counterexample where the condition  $f(0) \geq 0 $ set by Bellman fails but did not prove the  necessity of the condition  $f(0) \leq 0 $. This condition is proved in the current paper.

\item The condition that  $f'(x)$ is monotone increasing in Bellman treats those differentiable convex functions but the current result involves any convex function, like the absolute value, the maximum of linear functions, etc, as long as they satisfy the condition $f(0) \leq 0 $.
\end{itemize}

\section{A generalization of Weinberger's inequality with alternating terms}
\subsection{The class of functions $\mathcal{ W }$}
We introduce the class $\mathcal{ W}$ of functions defined as follows:
\begin{equation}\label{eq:generalclass}
 \mathcal{ W}:= \{f : [0,\infty) \rightarrow  \mathbb{R} \, \},
\end{equation}
satisfying the following two equivalent conditions:
\begin{equation}\label{eq:newcondition1}
  f(x)-f(y) \geq f(x-y) ,
\end{equation}
\begin{equation}\label{eq:newcondition2}
   f(x+y)\geq f(x)+f(y) ,
\end{equation}
for any $ x \geq y \geq 0 $.

\begin{theorem}
  The conditions (\ref{eq:newcondition1}) and (\ref{eq:newcondition2}) are equivalent.
\end{theorem}

\begin{proof}
 Let $x\geq y \geq 0 $. Suppose that $f(x+y) \geq f(x)+f(y)$. Then
  $$ f(x)= f((x-y)+ y) \geq f(x-y) + f(y) .$$
  from which we get $f(x) -f(y) \geq f(x)-f(y) $. On the other hand, if we assume that $f(x)-f(y) \geq f(x-y) $, then  $f(x+y)-f(y) \geq f(x) $. So $f(x+y) \geq f(x)+ f(y)$. This proves the Theorem.
\end{proof}

\begin{theorem}
The functions $f(x)=x^r, r> 1 $ in the Weinberger's result are in the class $\mathcal{W}$.
\end{theorem}

\begin{proof}
We show the condition in (\ref{eq:newcondition1}) holds true by proving the inequality :
  \begin{equation}\label{eq:powerinequality}
   x^r-y^r \geq (x-y)^r, \quad x \geq y \geq  0, \quad r > 1 .
  \end{equation}
  If $r \in \mathbb{N } $, then by binomial theorem considering the terms with $s=0$ and $s=r$, we have
  $$ x^r= ((x-y)+y )^r = \sum_{s=0}^{r}\binom{r}{s}(x-y)^sy^{r-s} \geq (x-y)^r +y^r ,  $$
  Consequently, inequality (\ref{eq:powerinequality}) follows. To complete the proof, we have to consider the case where $r>1$ is a non integer. We can write $r$ as sum of  the integer part $\lfloor r \rfloor $ and the fraction part $\{ r \}  $ as follows:
  $$  r= \lfloor r \rfloor + \{ r \} , $$
  where $\lfloor r \rfloor \in \mathbb{N}$ when $ r \geq  1 $, and $ 0 \leq     \{ r \} <  1  $. Then,
   \begin{align*}
     x^r-y^r &= x^{\lfloor r\rfloor}x^{\{ r \} } - y^{\lfloor r\rfloor}y^{\{ r \} }\\
      & = \left(x^{\lfloor r\rfloor}- y^{\lfloor r\rfloor} \right )x^{\{ r \} } + y^{\lfloor r\rfloor}\left( x^{\{ r \} }-y^{\{ r \} }  \right) \\
      & \geq  \left(x^{\lfloor r\rfloor}- y^{\lfloor r\rfloor} \right )(x-y)^{\{ r \} } + y^{\lfloor r\rfloor}\left( x^{\{ r \} }-y^{\{ r \} }  \right) \\
      & \geq  \left(x- y \right )^{\lfloor r\rfloor} (x-y)^{\{ r \} } + y^{\lfloor r\rfloor}\left( x^{\{ r \} }-y^{\{ r \} }  \right) \\
      & = (x-y)^r +  y^{\lfloor r\rfloor}\left( x^{\{ r \} }-y^{\{ r \} }  \right) \\
      & \geq (x-y)^r.
     \end{align*}
    Alternatively, we  may show the condition (\ref{eq:newcondition2}) by showing:
   \begin{equation}\label{eq:sumexpansion}
  (x+y)^r  \geq x^r +y^r, \quad x >0,\quad y > 0, \quad r \geq 1 .
  \end{equation}
Since the case $x = y= 0 $ is trivial, let $x>0$. Then
     $$  (x+y)^r =x^r (1+y/x)^r \geq x^r(1+y/x) = x^r +yx^{r-1} \geq x^r +y^r. $$
This completes the proof.
\end{proof}

\begin{theorem}
  Let $f, g \in \mathcal{ W} $, and let $\alpha > 0 $. Then $  \alpha f, f+g, fg,  \in \mathcal{ W} $.
\end{theorem}

\begin{proof}
 As the consequence of the condition (\ref{eq:newcondition2})  of functions in the class $\mathcal{ W}$, we have the following inequalities
 $$ (\alpha f) (x+y) = \alpha f(x+y) \geq \alpha (f(x)+f(y)) =(\alpha f)(x) +  (\alpha f)(y).     $$
 $$ (f+g)(x+y) =f(x+y)+g(x+y) \geq f(x)+f(y)+g(x)+g(y)= (f+g)(x)+(f+g)(y), $$
 $$ (fg)(x+y) =f(x+y)g(x+y) \geq (f(x)+f(y))(g(x)+g(y))\geq  (fg)(x)+(fg)(y).    $$
\end{proof}

\begin{theorem}\label{eq:generalfun}
  $ \mathcal{W} $ contains  all function of the form
  \begin{equation}\label{eq:generalfunctionform}
     f(x)=  \sum_{i=1 }^{m} a_{i} x^{r_i},
  \end{equation}
  where $a_i > 0,r_i > 1, \,  i=1,2,...,m $.
\end{theorem}

\begin{theorem}
  Let $$  f(x)= \sum_{i=0}^{\infty}  c_i f_i(x),\quad c_i \geq 0, \quad  f_i(x) \in \mathcal{W} ,      $$
  is an infinite series of functions convergent on some interval of the form $[0, a]$. Then $f \in\mathcal{ W} $.
\end{theorem}
\begin{proof}
  For $ a  \geq     x    \geq y  \geq 0 $, we have,
  $$   f(x)-f(y) = \sum_{i=0}^{\infty}  c_i ( f_i(x)-f_i(y)) \geq \sum_{i=0}^{\infty}  c_i  f_i(x-y)=  f(x-y))  $$
\end{proof}

\begin{example}

\end{example}
The  function $f(x)= e^x-x-1$ which has a Taylors series representation
  $$ f(x)= e^x-x-1 = \sum_{k=2}^{\infty}\frac{x^k}{k!}, $$
  that converges to $f(x)$ for every $x \in \mathbb{R }$,  has each of the terms $x^k /k!$ in the class $\mathcal{W}$. So $ f \in \mathcal{W} $.


\begin{theorem}\label{eq:nonnegativity}
  For any $f \in \mathcal{W} $, we have the following properties:
  \begin{itemize}
    \item $f(0)\leq 0$
    \item If $f$ is non negative then $f$ is nondecreasing.
  \end{itemize}
  \end{theorem}

  \begin{proof}
  \begin{itemize}
    \item By condition (\ref{eq:newcondition1}), we have
    $$f(0)= f(x-x)\leq f(x)-f(x)=0. $$

    \item  Let $x \geq y \geq 0 $. By condition (\ref{eq:newcondition2}) and nonnegativity of $f$ we have
      $$f(x)= f((x-y)+y)\geq f(x-y) +f(y) \geq f(y).$$
     Therefore $f$  is nondecreasing.
  \end{itemize}
\end{proof}

\begin{remark}
   By Theorem \ref{eq:nonnegativity}, if $f$ is a non negative  differentiable function in the class $\mathcal{W}$, then $f'(x) \geq 0 $. However, not all functions in the class $\mathcal{W}$ are differentiable.
\end{remark}

\begin{theorem}
  The identity function $f(x)=x$  is in the class $\mathcal{W}$ with equality holds in conditions (\ref{eq:newcondition1}) and (\ref{eq:newcondition2}). The constant function $f(x)= c $ is in $\mathcal{W}$ if and only if $c \leq  0 $.
\end{theorem}

\begin{theorem}
  The floor function $f(x)= \lfloor x \rfloor $ is in the class $\mathcal{W}$.
\end{theorem}
\begin{proof}
  For the floor function, we have the following property
  $$ \lfloor x \rfloor  +  \lfloor  y \rfloor  \leq x+y.                  $$
  Consequently,
$$  \lfloor x \rfloor  +  \lfloor  y \rfloor =   \lfloor   \lfloor x \rfloor  +  \lfloor  y \rfloor \rfloor   \leq \lfloor x+y \rfloor. $$
\end{proof}
\begin{remark}
  The floor function $f(x)= \lfloor x \rfloor $ is an example of a function that not convex, not differentiable, and not continuous. It is nondecreasing, and satisfies the condition $f(0)=0 .$
\end{remark}

\begin{theorem}
  Let $f, g \in \mathcal{W }$ are nonnegative. Then the composition $ f \circ g \in \mathcal{W} $.
\end{theorem}
\begin{proof}
By Theorem \ref{eq:nonnegativity}  both $f\in \mathcal{W }  $ and $g \in \mathcal{W }$ are  non decreasing functions. We get:
$$ f(g(x))-f(g(y))  \geq f(g(x)-g(y))    \geq f(g(x-y))   $$
 This shows that $f \circ g $ satisfies condition (\ref{eq:newcondition1}) and hence in $ \mathcal{W}$.
\end{proof}


\begin{theorem}
  Let $f$ be any function that satisfies the generalized Weinberger's inequality   (\ref{eq:generalizedinequality}). Then $ f \in \mathcal{W }$
\end{theorem}
\begin{proof}
  Let $f$ satisfy the inequality (\ref{eq:generalizedinequality}). Then fixing $k=2$, we get  the condition (\ref{eq:newcondition1}) of $\mathcal{W}$. This proves the Theorem.
\end{proof}

\subsection{Proof of Szeg\"{o}'s Inequality}
\begin{proof}
  The condition (\ref{eq:nonnegativedomaincondition}) was considered with  all strict inequalities in Szeg\"{o} as cited in \cite{IO}.   For notational brevity  let us  use:
 \begin{equation}\label{eq:brevity}
     \sum_{s=1}^{m}(-1)^{s-1} a_s := S_m(\tilde{a}_s),\quad   \sum_{s=1}^{m}(-1)^{s-1} f(a_s):= S_m( \tilde{f}(a_s))
 \end{equation}

We use induction. For $m=1$, the inequality $f(a_1) \leq f(a_1)$ holds trivially. Next we prove the inequality for $m=2$. From the assumption (\ref{eq:nonnegativedomaincondition}) we have that
$$ 0 \leq a_3 \leq a_1-a_2+a_3 \leq a_1   $$
It may happen that either
\begin{equation}\label{eq:3termcase1}
  a_3 \leq a_2 \leq a_1-a_2+a_3 \leq a_1
\end{equation}
or
\begin{equation}\label{eq:3termcase2}
  a_3  \leq a_1-a_2+a_3    \leq a_2       \leq a_1
\end{equation}
By the convexity of $f$,
\begin{equation}\label{eq:3termconvexity}
  \frac{f(a_1-a_2+a_3)-f(a_3)}{a_1-a_2} \leq \frac{f(a_1)-f(a_2)}{a_1-a_2}.
\end{equation}
Therefore,
\begin{equation}\label{eq:3terminequality}
  f(a_1-a_2+a_3)-f(a_3) \leq f(a_1)-f(a_2)+f(a_3).
\end{equation}
Now assume that the inequality holds for $m=k$. That is,
\begin{equation}\label{eq:inductionassumption}
  f(S_{2k-1}(\tilde{a}_s)) \leq S_{2k-1}( \tilde{f}(a_s)).
\end{equation}
Then we have
\begin{equation}\label{eq:2kplus1thsum}
   S_{2k+1}(\tilde{a}_s)   =   S_{2k-1}(\tilde{a}_s) -a_{2k} + a_{2k+1}
\end{equation}
From (\ref{eq:2kplus1thsum}) and (\ref{eq:nonnegativedomaincondition})  it follows that
\begin{equation}\label{eq:lastarrangement}
  S_{2k-1}(\tilde{a}_s) \geq a_{2k} \geq a_{2k+1}
\end{equation}
From (\ref{eq:lastarrangement}) and the result we have proved for $m=2$, we have
\begin{align*}
  f( S_{2k+1}(\tilde{a}_s) )  & \leq f( S_{2k-1}(\tilde{a}_s))-f(a_{2k})+f(a_{2k+1}) \\
   & \leq   S_{2k-1}( \tilde{f}(a_s))-f(a_{2k})+f(a_{2k+1})  \\
   & =  S_{2k+1}( \tilde{f}(a_s))
\end{align*}

Hence the Theorem is proved.
\end{proof}

Szeg\"{o}'s inequality, whose subjects are convex functions, is valid for inequalities involving odd number of terms. But it does not treat inequalities with even number of terms. In the next Theorem, we see that the case of even number of terms is resolved if the function $f \in \mathcal{W} $ in addition to its being convex.
\begin{theorem}\label{eq:convexplusw}
  For any  convex function $ f \in \mathcal{W} $ we have
  \begin{equation}\label{eq:generalizationequation}
   f( S_{2m}( \tilde{a}_s) )   \leq S_{2m}( \tilde{f}(a_s))
  \end{equation}
\end{theorem}

\begin{proof}
  For $m=1$, the result follows from the definition of $\mathcal{W}$. For $m> 1$,  by  Szsg\"{o}'s inequality and the assumption that $ f \in \mathcal{W }$, we  have
\begin{align*}
   f( S_{2m}( \tilde{a}_s) )  & =f(a_1-a_2+a_3-a_4+...+a_{2m-1}-a_{2m}) \\
     & =f((a_1-a_2+a_3-a_4+...+a_{2m-1})-a_{2m})  \\
     &  \leq f(a_1-a_2+a_3-a_4+...+a_{2m-1}) -f(a_{2m})\\
     & \leq f(a_1)- f(a_2) +f( a_3) -f( a_4)+...+ f(a_{2m-1}) -f(a_{2m})\\
     & = S_{2m}( \tilde{f}(a_s))
     \end{align*}
Hence the theorem is proved.
\end{proof}
\noindent We have seen that every function $f \in \mathcal{W} $ satisfies the conation $ f(0) \leq 0 $. In the next theorem we see that every convex function that satisfies the condition $ f(0) \leq 0 $ is in $ \mathcal{W} $

\begin{theorem}\label{eq:convexplusIC}
  Let $f$ be a convex function with $f(0) \leq 0 $. Then $f \in\mathcal{ W} $.
\end{theorem}

\begin{proof}
  Let $x  \geq y \geq 0  $. We use the assumption that $f$ is convex and consider two possible cases. The first case is $0 \leq x-y \leq y \leq x $. In this case,  by the condition of convexity of $f$, we have
  \begin{equation}\label{eq:case1}
      \frac{f(x-y)-f(0)}{x-y}\leq  \frac{f(x)-f(y)}{x-y}.
  \end{equation}
The second case is $0 \leq y \leq x-y \leq x $. In this case, by the condition of convexity of $f$, we have

  \begin{equation}\label{eq:case2}
    \frac{f(y)-f(0)}{y}\leq  \frac{f(x)-f(x-y)}{y}.
  \end{equation}
  Both inequalities (\ref{eq:case1}) and (\ref{eq:case2}) yield:
  \begin{equation}\label{eq:essentialfofzero}
    f(x-y)-f(0) \leq f(x)-f(y).
  \end{equation}
  If $f(0) \leq 0 $, by (\ref{eq:essentialfofzero}), we have $f(x-y) \leq f(x)-f(y)  $. This implies that $f \in \mathcal{W} $.
\end{proof}
The following theorem is the generalization of the Weinberger's inequality with alternating signs. It is not  restricted to convex functions of the form $f(x)=x^r, r> 1 $.  Due to the additional condition $f(0) \leq 0 $, the inequality is not restricted to odd number of terms like Sezeg\"{o}'s inequality.
\begin{theorem}
  Assume condition given in (\ref{eq:nonnegativedomaincondition}). If $f$ is a convex function with the condition $f(0) \leq 0 $. Then for any $m \in \mathbb{N} $, we have
  $$ f(S_{m}(\tilde{a}_s) \leq S_m( \tilde{f}(a_s)) $$
\end{theorem}

\subsection{ A Counterexamples to some perviously proposed generalized inequality  }

\begin{example}
Consider the function $f(x) = e^x $. Then $f$ satisfies all the conditions in (\ref{eq:Bellmanscondition}) set by Bellman \cite{RB}. Let $a_1= 1, a_2=0.1 $. We have
$$f(a_1- a_2)=f(0.9)= 2.4596 > 1.61311 =f(1)-f(0.1). $$
This contradicts the proposed generalized inequality (\ref{eq:generalizedinequality}).
\end{example}

\section{Summary notes}
\begin{itemize}
\item The class $\mathcal{W}$ is the set of  all functions $f$ that satisfy  the generalized inequality (\ref{eq:generalizedinequality}) exactly for  two terms. That is $k=2$.
\item $f(x)= \lfloor x \rfloor $ in the class $ \mathcal{W} $. But it is not convex. As a result it fails to satisfy condition in Theorem \ref{eq:convexplusw}. To see this, take $ a_1= 4.6, a_2 =3.1, a_3 =2.8, a_4=1.2 $.
\item The class $\mathcal{W}$ is closed under addition, multiplication,  scaler multiplication  by positive numbers. The subset of $\mathcal{W}$  that constitute all non negative functions is closed under  the composition of functions. Therefore $\mathcal{W}$ is an infinite set of functions or even uncountably infinite set of functions.
  \item Every element $f\in \mathcal{W }$ satisfies the condition $f(0) \leq 0 $.
\item Every nonnegative function $f\in \mathcal{W }$ is nondecreasing. For example, $ f(x) =x^r, r> 1 $, that are subjects of Weinberger's inequality.

\item The function
$$ f(x)=
    \begin{cases}
       & x \ln x,            \mbox{  if  } x > 0, \\
       &   0,                  \mbox{ if } x  = 0
    \end{cases} $$
    is convex on $[0, \infty )$. With additional condition $f(0)=0 $, it is in $\mathcal{W}$. It is negative on $(0,1)$ and it is increasing on $[1,\infty)$, where it takes nonnegative values.
\item  Convex function, for example, $f(x)=e^x $  satisfy Szeg\"{o}'s inequality (\ref{eq:szegoinequality}). However it is not in the class $\mathcal{W}$. As a result it fails to satisfy the inequality (\ref{eq:generalizationequation}) .
\item  The set all function $f$ that generalize Weinberger's inequality are the ones which are in the intersection of the set of convex functions and $\mathcal{W}$.
\end{itemize}

%
%
%
%

\section*{Conflict of interests}
The author declare that there is no conflict of interests regarding the publication of this paper.

\section*{Acknowledgment}
The author is thankful to Professor Albert Erkip (Doctoral advisor) for his continuous support and encouragement in his intellectual progress.
\section*{Data Availability}
The is no external data used in this paper.
\section*{Funding} This Research work is not funded by any institution or person.

\end{document}